\documentclass[english,12pt]{article}

\usepackage{amsmath,amssymb,amsthm}
\usepackage{latexsym}
\usepackage[pdftex]{graphicx}
\usepackage{float}
\usepackage{indentfirst}
\usepackage{color}
\usepackage{ulem}

\newtheorem{tm}{Theorem}

\newtheorem{kor}{Corollary}

\newtheorem{rem}{Remark}
\makeatother

\usepackage{babel}
\makeatother

\begin{document}

\vspace*{0.5cm}

\begin{center}
{\Large LDPC codes constructed from cubic symmetric graphs}
\end{center}

\vspace*{0.5cm}

\begin{center}

Dean Crnkovi\'c \\
({\it\small E-mail: deanc@math.uniri.hr})\\[3pt]
Sanja Rukavina \\
({\it\small E-mail: sanjar@math.uniri.hr})\\[3pt]
and \\[3pt]
Marina \v Simac\\ 
({\it\small E-mail: msimac@math.uniri.hr})\\[3pt]

\end{center}

\begin{center}
Department of Mathematics\\
 University of Rijeka\\
 Radmile Matej\v ci\'c 2, 51000 Rijeka, Croatia 
\par\end{center}



\vspace*{0.2cm}

\begin{abstract}
Low-density parity-check (LDPC) codes have been the subject of much interest due to the fact that they can perform near the Shannon limit.
In this paper we present a construction of LDPC codes from cubic symmetric graphs. The constructed codes are $(3,3)$-regular and the vast majority of the 
corresponding Tanner graphs have girth greater than four.
We analyse properties of the obtained codes and present bounds for the code parameters, the dimension and the minimum distance. 
Furthermore, we give an expression for the variance of the syndrome weight of the constructed codes. 
Information on the LDPC codes constructed from bipartite cubic symmetric graphs with less than 200 vertices is presented as well.
Some of the constructed codes are optimal, and some have an additional property of being self-orthogonal or linear codes with complementary dual (LCD codes).
\end{abstract}

\vspace*{0.2cm}


\textit{MSC $2020$ Codes\/:} 94B05, 05C99

\newpage

\section{Introduction and preliminaries}

We assume that the reader is familiar with the basic facts of graph theory and coding theory. 
We refer the reader to \cite{Balakrishnan, Diestel} and \cite{Pless} for related background materials on graphs and codes, respectively.
In this paper we consider only non-trivial finite connected graphs without loops and multiple edges.

A binary $[n,k]$ linear code $\mathcal{C}$ is a $k$-dimensional subspace of the vector space $\mathbb{F}_2^n$.   For $x,y \in \mathbb F_2^n$, the number $d(x,y)=|\{i : x_i\neq y_i, \ 1\leq i \leq n\}|$ 
is called the Hamming distance. The minimum distance of a code $\mathcal{C}$ is the number $d=min\{d(x,y):x,y \in \mathcal{C} \} $. Through the paper, a binary $[n,k]$ 
linear code with the minimum distance $d$ will be called an $[n,k,d]$ code. An optimal code is a code which achieves the theoretical upper bound for the minimum distance.

The codewords of an $[n,k,d]$ code satisfy $m\geq n-k$ parity-check equations. Every parity-check equation can be presented as a binary vector of length $n$ having $j$-th position equal to $1$ 
if  the corresponding codeword bit is included in that parity-check equation. The set of $m$ parity-check equations can be presented with an $m\times n $ parity-check matrix $H=\left[ h_{i,j} \right]$.
If $h_{i,j}=1$, then the $i$-th parity-check equation contains the $j$-th codeword bit. The rows of the parity-check matrix span the null space (or dual code) $\mathcal{C}^\perp$ of $\mathcal{C}$.

A binary low-density parity-check (LDPC) code is a binary linear code defined by a sparse parity-check matrix $H$, which means that $H$ contains a very small
number of nonzero entries. An LDPC code is called $(w_c, w_r)-$regular if $H$ has constant row sum $w_r$ and constant column sum $w_c$. 

An LDPC code can be presented using the Tanner graph, which gives a relation between parity-check equations and codeword bits. 
The Tanner graph is a bipartite graph that consists of two sets of vertices: bit nodes that correspond to codeword bits and check nodes that correspond to  parity-check equations. 
An edge connects a bit node to a check node if that bit is included in the corresponding parity-check equation. 
If an LDPC code is $(w_c, w_r)-$regular, then each bit node has degree $w_c$ and each check node has degree $w_r$.
A cycle in a graph is a sequence of edges that form a path in the graph such that the first node is equal to the last one. 
The length of a cycle is the number of edges in it, and the girth of a graph is the length of the shortest cycle. 
Since a Tanner graph is bipartite, the length of a cycle must be even and at least four. 

The decoding performance of an LDPC code depends on the structure of the corresponding Tanner graph. 
The existence of short cycles in the Tanner graph of a code establishes a correlation between iterations in the process of decoding, and thus, has a negative impact on the bit error rate 
(BER) performance of the code. The shorter the cycles are, the more significant the effect is. For this reason, the aim is to construct LDPC codes without short cycles, especially cycles of length four.

LDPC codes were first introduced by Gallager in the early 1960's (see \cite{Gallager}) and rediscovered by MacKay and Neal (see \cite{MacKay1}). These codes have been
the subject of much interest due to the fact that they can perform near the Shannon limit (see \cite{Gallager}). 
For some recent results on LDPC codes we refer the reader to \cite{Singh, Xu}.
Over the past years researchers have constructed LDPC codes that are free of cycles of length four using various structures, 
including graphs (see, e.g., \cite{geodetic, Rosenthal}). Regular bipartite graphs with large girth constructed in \cite{Lazebnik} were used in \cite{Kim} as Tanner graphs of LDPC codes.
In this paper we construct LDPC codes using bipartite cubic symmetric graphs as Tanner graphs of LDPC codes.

The paper is organized as follows: in Section \ref{section_cubic}, the construction of the LDPC codes using cubic symmetric graphs is introduced, and the results about the code parameters are presented. 
In Section \ref{section_variance}, an expression for the variance of a syndrome weight of the constructed LDPC codes is obtained.
In Section \ref{section_results}, computational results and constructed LDPC codes are presented.
Finally, in Section \ref{section_simulation} we give an example that illustrates the BER performance of the constructed codes over a binary symmetric channel.

\section{LDPC codes constructed from cubic symmetric graphs} \label{section_cubic}

Cubic graphs are 3-regular graphs, i.e. graphs in which all vertices have degree equal to three. A graph is symmetric if it is arc-transitive, i.e. if its automorphism group acts transitively on 
the set of arcs. 
Therefore, cubic symmetric graphs (CSGs) are 3-regular arc-transitive graphs.
CSGs were first studied by Foster in \cite{Foster}. They have since been the subject of much interest and study.
Conder and Nedela proved (see \cite{Conder}) that finite symmetric cubic graphs can be classified into 17 different families according to the arc-transitive actions they admit. 
The majority of CSGs are bipartite  and it is known that there exist exactly five connected CSGs with girth less than six (the complete graph $K_4$, the complete bipartite graph $K_{3,3}$, 
the cube, the Petersen graph and the dodecahedron).
In this paper we study LDPC codes having bipartite CSGs as the Tanner graphs. 

Let $ \mathcal G$ be a connected CSG with $2n$ vertices. Denote by $A$ its adjacency matrix. If $\mathcal G$ is a bipartite graph, its adjacency matrix can be written as follows:

	\begin{equation} \label{A}
	A=
	\left[\begin{array}{ c c } 
	0 & H \\
	H^T & 0 \\
		\end{array}  \right],  	\end{equation} where $H$ is an $n\times n$  matrix. \\
One can construct an LDPC code $\mathcal C( \mathcal G)$ by taking the  matrix $H$ as a parity-check matrix of the code. That is to say, $\mathcal G$ is the Tanner graph of the code 
$\mathcal C( \mathcal G)$. 
The density of the parity-check matrix is equal to $\displaystyle \frac 3 n$ and the obtained code is a $(3,3)$-regular LDPC code of length $n$ and dimension $n-rank_2(H)$, 
where $\displaystyle rank_2(H)=\frac{1}{2}rank_2(A)$.
Every arc-transitive graph without isolated vertices is vertex-transitive, so it is possible to obtain $H$ from $H^T$ by permuting the rows and columns.
Hence, the LDPC codes obtained from $H$ and $H^T$ are equivalent. Note that two binary codes are equivalent if and only if they are isomorphic.

If $\mathcal G$ is a non-bipartite CSG, its adjacency matrix $M$ determine a parity-check matrix of a $(3,3)$-regular LDPC code whose Tanner graph is a CSG having the 
adjacency matrix of the form (\ref{A}) with $H=M$. Hence, and according to the classification of CSGs, an LDPC code constructed from a non-bipartite CSG 
(taking its adjacency matrix as a parity-check matrix of the code) is isomorphic to the LDPC code obtained from some bipartite CSG  
(that is the Tanner graph of the code) with twice a number of vertices than the initial graph. Therefore, only LDPC codes constructed from bipartite CSGs will be considered.

 In the sequel, when considering a CSG $\mathcal G$, we refer to a connected bipartite CSG $\mathcal G$ with $2n$ vertices. 
By $\mathcal C( \mathcal G)$ we denote the LDPC code of length $n$ having $\mathcal G$ as its Tanner graph. 

Let $H$ be an $n\times n$  parity-check matrix of a code $\mathcal C( \mathcal G)$. A bit node  graph $\Gamma$ is defined in the following way: 
it has $n$ vertices that correspond to codeword bits, and two vertices are adjacent if and only if the corresponding bits are included in the same parity-check equation. 
In other words, two vertices of the graph $\Gamma$ are adjacent if and only if the corresponding bit nodes of the Tanner graph $\mathcal G$ of the code $\mathcal C( \mathcal G)$ have a common neighbour.

\begin{tm} \label{matG1}
Let $\mathcal G$ be a connected bipartite CSG with $2n$, $n \geq 7$, vertices and let $H$ be the parity-check matrix of the code $\mathcal C( \mathcal G)$. Then the corresponding bit node graph 
$\Gamma$ is 6-regular.
\end{tm}
\begin{proof}
	Every bipartite CSG $\mathcal G$ with  $2n$,  $n \geq 7$, vertices has the girth at least six.
	A bit node $v$ of the Tanner graph $\mathcal G$ of the code $\mathcal C( \mathcal G)$ has degree equal to three, and each of its neighbours is adjacent to another two bit nodes. 
	Since $\mathcal G$ does not have cycles of length four, the node $v$ has a common neighbour with exactly six other bit nodes. 
	In other words, the node $v$ of $\Gamma$ has degree equal to six.  Hence, the graph $\Gamma$ is 6-regular. 
\end{proof}

\begin{tm} \label{matG2}
	Let $\mathcal G$ be a connected bipartite CSG with $2n$, $n \geq 7$, vertices. Let $H$ be the parity-check matrix of the code $\mathcal C( \mathcal G)$ and let $\Gamma$ be the corresponding bit node graph. Then a $(0,1)$-matrix $T$ of order $n$ is the adjacency matrix of the graph $\Gamma$ if and only if $H^TH=3I+T$.
\end{tm}
\begin{proof}
	The diagonal elements of the matrix $H^TH$ correspond to the degree of bit nodes of the Tanner graph $\mathcal G$ of the code $\mathcal C( \mathcal G)$, which is equal to three.
	Since $\mathcal G$ does not have cycles of length four, the other elements of the matrix are 1 or 0 depending whether the corresponding bit nodes have a common neighbour or not. 
	That is to say, off-diagonal elements of the matrix $H^TH$ are 1 or 0 depending whether the corresponding nodes of the graph $\Gamma$ are adjacent or not. 
	Hence, $H^TH=3I+T$, where $T$ is the adjacency matrix of the graph $\Gamma$.	
	
	Conversly, suppose that $T$ is a $(0,1)$-matrix of order $n$ such that $H^TH=3I+T$. $H^TH$ is a symmetric matrix, and therefore $T$ is a symmetric matrix. Clearly, the matrix $T$ has zeroes on the diagonal. An off-diagonal element of $T$ corresponds to the number of common neighbours of the  corresponding bit nodes of the Tanner graph $\mathcal G$ of the code $\mathcal C( \mathcal G)$. 
Since the girth of the Tanner graph $\mathcal G$ is at least six, the number of common neighbours is zero or one. Therefore, $T$ is the adjacency matrix of the graph $\Gamma$.
\end{proof}

The following results can be found in \cite{Ryan}.
\begin{tm} \cite[Theorem 3.1]{Ryan}  \label{words}
Let $\mathcal C$ be a binary linear code with a parity-check matrix $H$. Then there exists a codeword in $\mathcal C$ with weight $w$ if and only if there are $w$ columns in $H$ whose vector sum is a zero vector.
\end{tm}
\begin{tm} \cite[Theorem 3.2]{Ryan}  \label{columns}
Let $\mathcal C$ be a binary linear code with a parity-check matrix $H$. Then the minimum distance of the code $\mathcal C$ is equal to the smallest number of columns in $H$ whose vector sum is a zero vector.
\end{tm}

Due to the fact that the column weight of a parity check matrix $H$ of a code $\mathcal C (\mathcal G)$ is equal to three, Theorem \ref{words} implies 
that the code $\mathcal C (\mathcal G)$ is even. Moreover, the minimum distance of the code $\mathcal C (\mathcal G)$ is an even number.  \\

We will use Theorem \ref{columns} in the proof of the following theorem. 

\begin{tm} \label{mindG}
Let $\mathcal G$ be a connected bipartite CSG with $2n$, $n \geq 7$, vertices  and let  $\Gamma$ be the corresponding bit node graph. 
The minimum distance of the code $\mathcal C( \mathcal G)$ is at least six if and only if the clique number of the graph $\Gamma$ is at most three.
\end{tm}
\begin{proof}
The girth of the Tanner graph $\mathcal G$ of the code $\mathcal C( \mathcal G)$ is  greater than four and, therefore, the minimum distance of the code is at least four (see \cite{Storme}).
 
If there exists a set $S$ in the graph $\mathcal G$ which consists of four bit nodes with the property that every pair of the vertices has a common neighbour, 
then the sum of the corresponding columns of the parity-check matrix equals zero. Hence, the minimum distance is equal to four. 
In other words, if the complete graph $K_4$ is the subgraph of the graph $\Gamma$, then the minimum distance of the code is equal to four.  
Consequently, if the minimum distance  of the code is at least six, then the clique number of the graph $\Gamma$ is at most three.

Conversly, assume that the clique number of the graph $\Gamma$ is at most three. Accordingly,  it is not possible to find four columns of the parity-check matrix whose sum equals zero. 
As the result of Theorem \ref{columns}, the minimum distance of the code is at least six.
\end{proof}

In \cite{Tanner} Tanner gave a lower bounds for the minimum distance of a regular LDPC code with a parity-check matrix $M$ in terms of the  
second largest eigenvalue $\mu_2$ of the matrix $M^TM$. 
\begin{tm}\cite[Theorems 3.1 and 4.1]{Tanner} \label{Tanner_d}
If the Tanner graph of a $(w_c,w_r)-$regular LDPC code is connected and has $n$ bit nodes, then the minimum distance of the code satisfies $d \geq max\{d_1, d_2\}$, where
$$d_1=\frac{n(2w_c-\mu_2)}{w_cw_r-\mu_2}, \ \ d_2 = \frac{2n(2w_c+w_r-2-\mu_2)}{w_r(w_cw_r-\mu_2)}.$$\\
\end{tm}

\begin{tm} \label{mind}

Let $\mathcal G$ be a connected bipartite CSG with $2n$, $n \geq 7$, vertices and let $\lambda_2$ be the second largest eigenvalue of its adjacency matrix $A$. 
Then the minimum distance of the code $\mathcal C( \mathcal G)$ satisfies the following condition:
$$ \displaystyle d\geq \left\{ \begin{aligned}{}
	\frac 2 5 n, &&   \lambda_2 \leq 2,\\
	\frac 2 9 n, && 2 < \lambda_2 \leq \sqrt 6, \\
	4, && \sqrt 6 <   \lambda_2< 3. \\
\end{aligned} \right .$$
\end{tm}
\begin{proof}
Since $\mathcal G$ is a bipartite graph, the spectrum of $\mathcal G$ is symmetric with respect to 0. Moreover, since $\mathcal G$ is 3-regular, 
for every eigenvalue $\lambda$ of its adjacency matrix $A$ the inequality $\left|\lambda\right|\leq 3$ holds (see \cite[Theorem 11.5.1]{Balakrishnan}). 
Furthermore, if $\lambda$ is an eigenvalue of $A$, then  $\lambda^2$ is an eigenvalue of $A^2$. 
Let $\mu_2$ be the second largest eigenvalue of the matrix $H^TH$, 
where $H$ is the parity-check matrix of the code $\mathcal C( \mathcal G)$. Using the fact that the matrices $H^TH$ and $HH^T$ have the same non-zero eigenvalues and the fact that $\lambda_2\geq1$ 
(see \cite{Nilli}), it follows that $\mu_2 = \lambda_2^2$.

According to Theorem \ref{Tanner_d}, the minimum distance of the code $\mathcal C( \mathcal G)$ satisfies $d \geq \textrm{max}\{d_1, d_2\}$, where
$$d_1=\frac{n(6-\mu_2)}{9-\mu_2}, \ \ d_2 = \frac{2n(7-\mu_2)}{3\cdot(9-\mu_2)}.$$ From this equalities one can easily obtain
	\begin{equation} \label{d1_mu2}
\mu_2=\frac{6n-9d_1}{n-d_1},
\end{equation}
and 
	\begin{equation} \label{d2_mu2}
\mu_2 = \frac{14n-27d_2}{2n-3d_2}.
\end{equation}

The inequality $d_1\geq d_2$ holds for $\mu_2 \leq 4$, and $d_2\geq d_1$ for $\mu_2 > 4$ (see \cite{Shibuya}). In the first case, when $\mu_2 \leq 4$, using the equality (\ref{d1_mu2}), one gets 
that $d\geq d_1 \geq \frac{2}{5}n$. 
If  $4 < \mu_2 \leq  6$, then the minimum distance satisfies the inequality $d\geq d_2 \geq \frac{2}{9}n$, which can be obtained using (\ref{d2_mu2}).
For $\mu_2 \geq 6$ the minimum distance of the code is at least four,  as discussed in the proof of Theorem \ref{mindG}.
\end{proof}

\begin{tm} \label{dimG}
Let $\mathcal G$ be a connected bipartite CSG with $2n$, $n \geq 7$, vertices. Then the dimension of the code $\mathcal C( \mathcal G)$ is at most $\displaystyle \frac{5n}{6}$.
\end{tm}
\begin{proof}
To determine the dimension of the code $\mathcal C( \mathcal G)$ we observe the 2-rank of its parity-check matrix $H$.   
The 2-rank of the matrix $H$ is greater or equal than the size of a maximal independent set of the bit node graph $\Gamma$. Hence, the dimension of the code is at most  $n- \alpha(\Gamma)$, 
where $n$ is the length of the code, i.e. the number of vertices of the graph $\Gamma$, and  $\alpha(\Gamma)$ is the independence number of the  graph $\Gamma$. 

According to \cite{Catlin} every 6-regular connected graph with $n$ vertices, other than $K_7$, has an independent set which contains at least $\frac{n}{6}$ vertices. 
So, it follows from Theorem \ref{matG1} that the inequality $\alpha(\Gamma)\geq\frac{n}{6} $ holds. Therefore, the upper bound for the dimension of the code is $\displaystyle \frac{5n}{6}$.
\end{proof}

If an adjacency matrix $A$ of a bipartite CSG $\mathcal G$ has the full rank, then the parity-check matrix $H$ of the LDPC code also has the full rank. Hence,  
the constructed LDPC code has the dimension zero, i.e. the constructed code is trivial. \\

We will need the following two results from \cite{Potocnik}. 
\begin{tm}\cite[Corollary 5]{Potocnik} \label{3-arc}
	The dimension of the nullspace of an adjacency matrix of a connected 3-arc-transitive graph which is $s$-regular for $s \ge 2 $ is non-zero.
\end{tm}
\begin{tm}\cite[Theorem 13]{Potocnik} \label{potocnik}
	Let $p$ be a prime number and let $\mathcal G$ be a vertex-transitive $s$-regular multigraph with $n$ vertices. Let $\mathbb F$ be a field of characteristic $p$. If $gcd(p, s)$ = 1 and $n$ is a power	of $p$, then the adjacency matrix of the graph $\mathcal G$ is invertible over $\mathbb F$.
\end{tm}

Theorem \ref{thm-trivial} gives a condition for the LDPC code constructed from a CSG with $v=2^t$ vertices to be trivial.

\begin{tm} \label{thm-trivial}
Let $\mathcal G$ be a connected bipartite CSG with $v=2^t$ vertices and let $\mathcal C( \mathcal G)$ be the LDPC code constructed from $\mathcal G$. 
Then the parameters of the code $\mathcal C( \mathcal G)$ are $[2^{t-1},0,2^{t-1}]$.
\begin{proof}
The length of the code is $\frac 1 2 v=2^{t-1}$. Every arc-transitive graph without isolated vertices is vertex-transitive, so $\mathcal G$ is a vertex-transitive graph. 
Since $\mathcal G$ is a 3-regular graph with $v=2^t$ vertices, using Theorem \ref{potocnik} one concludes that the adjacency matrix of $\mathcal G$ has the full rank over 
$\mathbb{F}_2$. Hence, the parity-check matrix of the code $\mathcal C( \mathcal G)$ is the full rank matrix, so the dimension of the code $\mathcal C( \mathcal G)$ is equal to 0. 
\end{proof}
\end{tm}

In the case when $\mathcal G$ is a connected bipartite 3-arc-transitive CSG, the dimension of the code $\mathcal C( \mathcal G)$ is greater than zero. 
This statement follows directly from Theorem \ref{3-arc}. 

The following theorem can be found in \cite{Chilappagari}. The Gallager A algorithm is also described in that reference.

\begin{tm} \cite[Theorem 2]{Chilappagari} \label{thm-error-correction}
A column-weight-three code with Tanner graph of girth $g \ge 10$ can correct $\frac{g}{2} -1$ errors in $\frac{g}{2}$ iterations of the Gallager A algorithm.
\end{tm}

The following statement is a direct consequence of Theorem \ref{thm-error-correction}.

\begin{kor} \label{cor-error-correction}
Let $\mathcal G$ be a bipartite CSG with the girth $g \ge 10$, and let $\mathcal C( \mathcal G)$ be the LDPC code constructed from $\mathcal G$. 
Then $\mathcal C( \mathcal G)$ can correct $\frac{g}{2} -1$ errors in $\frac{g}{2}$ iterations of the Gallager A algorithm.
\end{kor}

\section{The variance of a syndrome weight} \label{section_variance}

A channel state information (CSI), for example the crossover probability, 
is very important for communication systems and it can be used for predicting a decoding efficiency. 
To compute a syndrome, an observed channel is converted to a binary symmetric channel (BSC) and the CSI of the original channel is derived from the estimated crossover probability of the BSC.
The estimation (performed prior to decoding) of the crossover probability based on the probability of a syndrome weight was proposed in \cite{Pacher} and \cite{Toto}. 
A general expression for the variance of a syndrome weight of an LDPC code 
(after transmission over a BSC) is given in \cite{Simos}. Obtained results were applied for LDPC codes that have constant check node degree. 

In this section we give the expression for the variance of a syndrome weight of an LDPC code  constructed from a bipartite CSG.

Let $\mathcal C$ be a binary linear code and let an $m\times n$  matrix $H$ be its parity-check matrix with row weights $r_i, \ i \in \{1,\ldots,m\}$. 
Furthermore, let $HH^T=[\lambda_{i,j}]$ be the concurrence matrix of $H$.
Suppose a codeword $c\in \mathcal C$ has been sent through a BSC with crossover probability $\rho$ and suppose that a vector $y$ has been received. 
The vector $s=y\cdot H^T$, $s=(s_1,\ldots, s_m)$,  is the syndrome of $y$. Denote by $w$ the syndrome weight, 
i.e. $w=\displaystyle \sum _{i=1}^m s_i$. \\
Let $f_t$ be the function defined by: $$f_{t}(\rho)=\frac{1-(1-2\rho)^{t}}{2}.$$
For an LDPC code whose Tanner graph is free of cycles of length four, in the case when the check nodes have the same degree $r$, 
the variance of a syndrome weight $w$ can be calculated as follows (see \cite{Simos}): 
\begin{equation} \label{Var}
Var(w)=\frac{m}{2}f_{2r}(\rho)+\frac{g_1}{2}\left(f_{2r}(\rho)-f_{2r-2}(\rho)\right),
\end{equation} 
where $\displaystyle g_1=\sum_{i\ne j}\lambda_{i,j}$.

The entries $\lambda_{i,j}, \ i \ne j$, of the concurrence matrix of a parity-check matrix $H$ of an LDPC code $\mathcal C( \mathcal G)$ are elements of the set $\{0,1\}$.  
A value $\lambda_{i,j}$ presents the number of common neighbours for $i$-th and $j$-th check node of the Tanner graph. 
Hence, $\lambda_{i,j}=1$ if the corresponding check nodes have a common neighbour, and $\lambda_{i,j}=0$ otherwise.
Using simple counting, it can be seen that each check node has a common neighbour with exactly six check nodes. Accordingly, $g_{1}=6n$. 
Using the equality (\ref{Var}), the variance of a syndrome weight can be calculated as follows:
$$Var(w)=\frac{n}{2}\left(7f_6(\rho)-6f_4(\rho) \right).$$

\section{Computational results} \label{section_results}

In this section we present information on LDPC codes constructed from the bipartite cubic symmetric graphs with less than 200 vertices.
We have used cubic symmetric graphs available at \cite{CubicGraphs} and follow the given notation. 
The parameters of the constructed LDPC codes are given in Table \ref{codes}. 
The obtained codes have low rates and good minimum distance.

\begin{center}
\begin{table}[H]
\scalebox{0.77}{
 $\begin{array}{|c|c|c|}
\hline
\begin{array}{ccc}
\textbf{\textrm{CSG}}& \textbf{\textrm{LDPC}} & \textbf{\textrm{Girth}}\\
\hline
6A & \textbf{[3,2,2]}^* & 4 \\
8A & [4,0,4] &  4\\
14A & \textbf{[7,3,4]} &  6\\
16A & [8,0,8] &  6\\
18A & \textbf{[9,2,6]}^* & 6 \\
20B & \textbf{[10,4,4]}^* & 6  \\
24A & \textbf{[12,4,6]} & 6 \\
26A & [13,0,13] & 6  \\
30A & [15,5,6] & 8 \\
32A & [16,0,16] & 6 \\
38A & [19,0,19] & 6 \\
40A & [20,4,8] & 8 \\
42A & \textbf{[21,5,10]} &6  \\
48A & \textbf{[24,6,10]} & 8 \\
50A & [25,0,25] & 6 \\
54A & \textbf{[27,2,18]}^* & 6 \\
56A & \textbf{[28,6,12]} & 6 \\
56C & [28,8,8]^* & 8 \\
62A & [31,0,31] & 6 \\
64A & [32,0,32] & 8 \\
72A & \textbf{[36,4,18]} & 6 \\
\end{array}
&
\begin{array}{ccc}
\hline
\textbf{\textrm{CSG}}& \textbf{\textrm{LDPC}} & \textbf{\textrm{Girth}}\\
\hline
74A & [37,0,37] & 6 \\
78A & \textbf{[39,2,26]}^* & 6 \\
80A & [40,4,16] & 10 \\
86A & [43,0,43] & 6 \\
90A & [45,11,10] & 10 \\
96A & [48,8,18] & 6 \\
96B & [48,12,8] & 8 \\
98A & \textbf{[49,3,28]} & 6 \\ 
98B & \textbf{[49,6,24]} & 6 \\
104A & [52,0,52] & 6 \\
110A & [55,10,10]^* & 10 \\
112A & [56,8,14] & 8 \\
112B & [56,8,16] & 8 \\
112C & \textbf{[56,6,24]} & 10 \\
114A & \textbf{[57,2,38]}^* & 6 \\
120A & [60,5,20] & 8 \\
120B & [60,4,24] & 10 \\
122A & [61,0,61] & 6 \\
126A & [63,5,30] & 6 \\
128A & [64,0,64] & 6 \\
128B & [64,0,64] & 10 \\ 
\end{array}
&
\begin{array}{ccc}
\hline
\textbf{\textrm{CSG}}& \textbf{\textrm{LDPC}} & \textbf{\textrm{Girth}}\\
\hline
134A & [67,0,67] & 6 \\ 
144A & \textbf{[72,8,32]} &  8 \\
144B & [72,6,30] & 10 \\
146A & [73,9,28] & 6 \\
150A & \textbf{[75,2,50]}^* & 6 \\
152A & [76,0,76] & 6 \\
158A & [79,0,79] & 6 \\
162A & \textbf{[81,2,54]}^* & 6 \\
162B & \textbf{[81,2,54]}^* & 12 \\
162C & [81,14,18] & 12 \\
168A & [84,10,30] & 6 \\
168E & [84,13,30] & 12 \\ 
168F & [84,8,38]^* & 12 \\
182A &  \textbf{[91,3,52]} & 6 \\
182B & \textbf{[91,3,52] }& 6 \\
182D & [91,14,26]^* & 12 \\
186A & \textbf{[93,2,62]}^* & 6 \\
192A & [96,23,8] & 8 \\
192B & [96,16,22] & 10 \\
192C & [96,18,16] & 12 \\
194A & [97,0,97] & 6 \\
\end{array}  \\
\hline
\end{array}$}
\caption{\footnotesize Parameters of LDPC codes constructed from bipartite cubic symmetric graphs with less than 200 vertices.}
\label{codes}
\end{table}
\end{center}

The girths of the Tanner graphs of the constructed codes are at least six, except for the codes obtained from the graphs $6A$ and $8A$ (the complete bipartite graph $K_{3,3}$ and the cube, respectively) for which the girth is equal to four.    
The LDPC code constructed from the graph $14A$ is isomorphic to the LDPC code obtained from the projective plane of order two,  i.e. the symmetric $2$-$(7,3,1)$ design, 
by taking the incidence matrix of the projective plane as the parity-check matrix of the code.
The codes obtained from the graphs $162A$ and $162B$ are isomorphic, as well as the codes obtained from the graphs $182A$ and $182B$. 
Some of the constructed codes, which are marked in bold, achieve upper bound for the minimum distance, i.e., these codes are optimal codes.
Further, some of the constructed codes enjoy an additional property of being an self-orthogonal or an LCD code.

A linear code $\mathcal C$ satisfying $\mathcal C\subseteq \mathcal C^\bot$, where $\mathcal C^\bot$ is the dual code of the code $\mathcal C$, is called a self-orthogonal code.
Some of the obtained codes are self-orthogonal. The codes with this property are obtained from the following graphs: 
$14A$, $30A$, $40A$, $56A$, $80A$, $90A$, $98A$, $98B$, $112B$, $112C$, $120A,$ $120B$, $144A$, $146A$ and $182A$ (and $182B$).

An LCD  code (linear code with complementary dual) is a linear code $\mathcal C$ which satisfies $\mathcal C\cap \mathcal C^\bot=\{0\}$. 
LCD codes were introduced by Massey in \cite{Massey}.  These codes have an important role in cryptography. 
Lately, there has been much interest and a lot of work has been done regarding this topic (see, e.g., \cite{Carlet, Harada, Key}).
The codes labeled with $^*$ in Table \ref{codes} are LCD codes.

The codes constructed from the graphs $6A$, $18A$, $54A$, $78A$, $114A$, $150A$, $162A$ (and $162B$), and $ 186A$ are unique LCD codes with the given parameters, 
up to equivalence (see \cite[Proposition 2.5, Theorem 4.5]{Harada}). 
According to the classification of LCD codes given in \cite{Harada}, there exists exactly five LCD codes with the parameters $[10,4,4]$. In this paper,  
an $[10,4,4]$ LCD code was obtained using the adjacency matrix of the  cubic symmetric graph $20B$. 

From Corollary \ref{cor-error-correction} it follows that the codes from the CSGs $90A$, $110A$, $112C$, $120B$, $128B$, $144B$ and $192B$ 
can correct 4 errors in 5 iterations of the Gallager A algorithm, and the
codes from the graphs $162B$, $162C$, $168E$, $168F$, $182D$ and $192C$ can correct 5 errors in 6 iterations of that algorithm.

\begin{rem}
The obtained LDPC codes have small rate. To obtain higher rate codes one can do the following. If $A$ is the parity-check matrix of an $[n,k,d]$ LDPC code $\mathcal C$ constructed from a bipartite cubic 
symmetric graph, one can use the matrix $A'= [ \ A \ | \ I_n \ ]$ as the parity-check matrix of an LDPC code $\mathcal C'$, where $I_n$ is the identity matrix of order $n$. 
The code $\mathcal C'$ is an irregular LDPC code of length $2n$ and dimension $n$, i.e. with rate equal 0.5. The minimum distance of the code is $4$, and the girth of the Tanner graph is preserved, 
i.e. the girth of the Tanner graph of the obtained code is the same as the girth of the Tanner graph of the initial code. Instead of the matrix $I_n$ one can use an $n \times l$ matrix $B$ consisting
of $l$ columns of the matrix $I_n$ to obtain an LDPC code with the parity-check matrix $[ \ A \ | \ B \ ]$ that has length $n+l$ and dimension at least $k$. The minimum distance of the code is equal to 4 for $l\geq3$.
\end{rem}

\section{Simulation results} \label{section_simulation}

As an illustration, we present simulation results of the $[24,6,10]$ and $[96,18,16]$ LDPC codes, derived from the the cubic symmetric graphs $48A$ and $192C$, on the additive white gaussian noise (AWGN) channel. We have compared the codes with randomly generated LDPC codes of the same length and dimension and a parity-check matrix with a column weight equal to three.  For randomly generated codes we have used the software for LDPC codes available at \cite{Radford}, which employs the construction from \cite{MacKay1, MacKay2}.  The codes are decoded with the sum-product decoding algorithm and the maximum number of iteration is set to 50. 
Figures \ref{fig:[24,6]} and \ref{fig:[96,18]} show the performance of the codes.
 
\begin{figure}[H]
	\begin{center}
		\includegraphics[width=85mm]{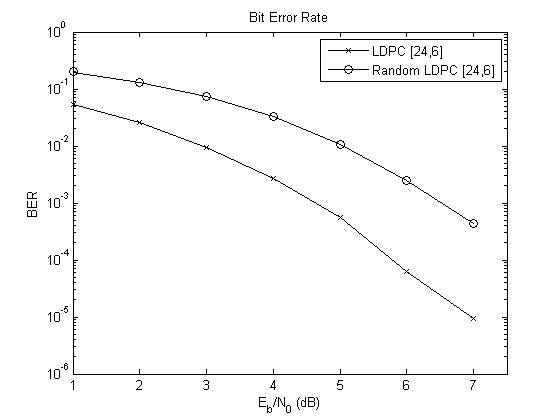}
		\caption{\small{BER performance of the $[24,6,10]$ LDPC code derived from the graph $48A$ }}
		\label{fig:[24,6]}
	\end{center}
\end{figure}

\begin{figure}[H]
	\begin{center}
		\includegraphics[width=85mm]{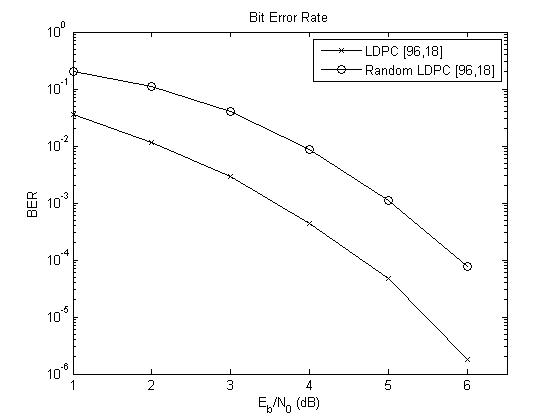}
		\caption{\small{BER performance of the $[96,18,16]$ LDPC code derived from the graph $192C$ }}
		\label{fig:[96,18]}
	\end{center}
\end{figure}

It can be seen from Figures \ref{fig:[24,6]} and \ref{fig:[96,18]}
that the LDPC codes constructed from the cubic symmetric graphs, comparing to randomly generated LDPC codes, 
have better BER performance.

%
%

\bigskip

\noindent {\bf Acknowledgement} 

This work has been fully supported by {\rm C}roatian Science Foundation under the project 6732.

\end{document}